\pdfoutput=1
\RequirePackage{ifpdf}
\ifpdf 
\documentclass[pdftex]{sigma}
\else
\documentclass{sigma}
\fi

\numberwithin{equation}{section}
\usepackage{enumerate}

\newtheorem{Theorem}{Theorem}[section]
\newtheorem{Problem}[Theorem]{Open Problem}
{\theoremstyle{definition}
\newtheorem{Definition}[Theorem]{Definition}
\newtheorem{Example}[Theorem]{Example}
\newtheorem{Remark}[Theorem]{Remark}
}

\def\PI{P$_{\rm I}$}
\def\PII{P$_{\rm II}$}
\def\PIII{P$_{\rm III}$}

\def\PIV{P$_{\rm IV}$}
\def\PV{P$_{\rm V}$}
\def\PVI{P$_{\rm VI}$}
\def\Ptf{$\mbox{\rm P}_{34}$}

\def\HII{\text{H$_{\rm II}$}}

\def\sPI{\text{S$_{\rm I}$}}

\def\sPVI{\text{S$_{\rm VI}$}}

\def\a{\alpha}
\def\b{\beta}

\def\ga{\gamma}
\def\de{\delta}
\def\k{\kappa}

\def\la{\lambda}

\def\vth{\vartheta}

\def\ph{\varphi}
\def\ds{\displaystyle}
\def\Ai{\mathop{\rm Ai}\nolimits}
\def\Bi{\mathop{\rm Bi}\nolimits}
\def\sgn{\mathop{\rm sgn}\nolimits}
\def\O{\mathcal{O}}
\def\C{\mathbb{C}}
\def\R{\mathbb{R}}

\def\d{{\rm d}}
\def\e{{\rm e}}
\def\i{{\rm i}}

\newcommand{\deriv}[3][]{\frac{\d^{#1}{#2}}{{\d{#3}}^{#1}}}
\newcommand{\ideriv}[3][]{{\d^{#1}{#2}}/{{\d{#3}}^{#1}}}
\newcommand{\pderiv}[3][]{\frac{\partial^{#1}{#2}}{{\partial{#3}}^{#1}}}

\begin{document}

\allowdisplaybreaks

\renewcommand{\thefootnote}{}

\renewcommand{\PaperNumber}{006}

\FirstPageHeading

\ShortArticleName{Open Problems for Painlev\'e Equations}

\ArticleName{Open Problems for Painlev\'e Equations\footnote{This paper is a~contribution to the Special Issue on Orthogonal Polynomials, Special Functions and Applications (OPSFA14). The full collection is available at \href{https://www.emis.de/journals/SIGMA/OPSFA2017.html}{https://www.emis.de/journals/SIGMA/OPSFA2017.html}}}

\Author{Peter A.~CLARKSON}

\AuthorNameForHeading{P.A.~Clarkson}

\Address{School of Mathematics, Statistics and Actuarial Science, University of Kent,\\ Canterbury, Kent, CT2 7FS, UK}
\Email{\href{mailto:P.A.Clarkson@kent.ac.uk}{P.A.Clarkson@kent.ac.uk}}
\URLaddress{\url{https://www.kent.ac.uk/smsas/our-people/profiles/clarkson_peter.html}}

\ArticleDates{Received January 18, 2019; Published online January 29, 2019}

\Abstract{In this paper some open problems for Painlev\'e equations are discussed. In particular the following open problems are described: (i)~the Painlev\'e equivalence problem; (ii)~notation for solutions of the Painlev\'e equations; (iii)~numerical solution of Painlev\'e equations; and (iv)~the classification of properties of Painlev\'e equations.}


\Classification{33E17; 34M55}

\renewcommand{\thefootnote}{\arabic{footnote}}
\setcounter{footnote}{0}

\section{Introduction}

The Painlev\'e equations are now regarded as ``nonlinear special functions'', being nonlinear analogs of the classical special functions and form the core of ``modern special function theo\-ry'' \cite{refPAC05review,refFIKN,refGLS,refUmemura98}. Indeed Iwasaki, Kimura, Shimomura and Yoshida \cite{refIKSY} characterize the Painlev\'e equations as ``the most important nonlinear ordinary differential equations" and state that ``many specialists believe that during the twenty-first century the Painlev\'e functions will become new members of the community of special functions". Subsequently this has happened as the Painlev\'e equations are a chapter in the NIST \textit{Digital Library of Mathematical Functions} \cite[Section~32]{refDLMF}. The Painlev\'e functions have greatly expanded the role that the classical special functions, such as the Airy, Bessel, Hermite, Legendre and hypergeometric functions, started to play in the 19th century. Increasingly, as nonlinear science develops, people are finding that the solutions to an extraordinarily broad array of scientific problems, from neutron scattering theory, special solutions of partial differential equations such as nonlinear wave equations, fibre optics, transportation problems, combinatorics, random matrices, quantum gravity and to number theory, can be expressed in terms of solutions of the Painlev\'e equations.

The Painlev\'e equations (\PI--\PVI), whose solutions are called the Painlev\'e transcendents, are the nonlinear ordinary differential equations given by
\begin{gather}
\deriv[2]{w}{z} = 6w^2+z,\label{eqPI}\\
\deriv[2]{w}{z} = 2w^3+zw+\alpha,\label{eqPII}\\
\deriv[2]{w}{z} = \frac{1}{w}{\left(\deriv{w}{z}\right)^{\!2}} - \frac{1}{z} \deriv{w}{z} +\frac{\a w^2 + \b}{z} + \ga w^3 + \frac{\de}{w},\label{eqPIII}\\
\deriv[2]{w}{z} =\frac{1}{2w}\left(\deriv{w}{z}\right)^{\!2} + \frac{3}{2}w^3 + 4z w^2 + 2(z^2 - \a)w + \frac{\b}{w},\label{eqPIV} \\
\deriv[2]{w}{z} = \left(\frac{1}{2w} + \frac{1}{w-1}\right){\left(\deriv{w}{z}\right)^{\!2}} - \frac{1}{z} \deriv{w}{z} + \frac{(w-1)^2}{z^2}\left(\a w +\frac{\b}{w}\right)+ \frac{\ga w}{z} + \frac{\de w(w+1)}{w-1},\label{eqPV} \\
\deriv[2]{w}{z} = \frac{1}{2}\left(\frac{1}{w} + \frac{1}{w-1} + \frac{1}{w-z}\right)\!\left(\deriv{w}{z}\right)^{\!2}
- \left(\frac{1}{z} + \frac{1}{z-1} + \frac{1}{w-z}\right) \!\deriv{w}{z} \nonumber \\
\hphantom{\deriv[2]{w}{z} =}{}+ \frac{w(w-1)(w-z)}{z^2(z-1)^2}\left\{\a + \frac{\b z}{w^2} + \frac{\ga (z-1)}{(w-1)^2} + \frac{\de z(z-1)}{(w-z)^2}\right\},\label{eqPVI}
\end{gather} where $\a$, $\b$, $\ga$ and $\de$ are arbitrary constants.
These six equations have attracted much attention for mathematicians and physicists during the past 40 years or so, though they were discovered by Painlev\'e, Gambier et al.\ in the late 19th and early 20th centuries, in an investigation of which second-order ordinary differential equations of the form
\begin{gather} \label{eq:gen-ode} \deriv[2]{w}{z}=F\left(w,\deriv{w}{z},z\right), \end{gather}
where $F$ is rational in $w$ and $\ideriv{w}{z}$ and locally analytic in $z$, having the property that their solutions have no movable branch points. They showed that there were fifty canonical equations of the form \eqref{eq:gen-ode} with this property, now known as the \textit{Painlev\'e property}, up to a M\"obius (bilinear rational) transformation
\begin{gather} \label{eq:mobius}
W(\zeta) = \frac{a(z)w + b(z)}{c(z)w + d(z)}, \qquad \zeta=\phi(z),
\end{gather} where $a(z)$, $b(z)$, $c(z)$, $d(z)$ and $\phi(z)$ are
locally analytic functions. Further Painlev\'e, Gambier et al.\ showed that of these fifty equations, forty-four can be reduced to linear equations, solved in terms of elliptic functions, or are reducible to one of six new nonlinear ordinary differential equations that define new transcendental functions, see Ince \cite[Chapter~14]{refInce}.

Following Sakai \cite{refSakai01} and Ohyama et al.\ \cite{refOKSO} (see also \cite{refOK06}), \PIII\ \eqref{eqPIII} can be classified into four cases:
\begin{enumerate}[(i)]\itemsep=0pt
\item if $\ga\de\not=0$, which is known as $\mbox{P}^{(6)}_{\rm III}$, then set $\ga=1$ and $\de=-1$, without loss of generality, by rescaling $w$ and $z$ if necessary
\begin{gather}\label{eq:PIII6}\deriv[2]{w}{z} = \frac1w \left(\deriv{w}{z}\right)^{2}- \frac{1}{z} \deriv{w}{z} +\frac{\a w^2}{z} + \frac{\b}{z} + w^3 - \frac{1}{w};\end{gather}
\item if $\ga=0$ and $\a\de\not=0$ (or equivalently $\de=0$ and $\b\ga\not=0$), which is known as $\mbox{P}^{(7)}_{\rm III}$, then set $\a=1$ and $\de=-1$, without loss of generality
\begin{gather}\label{eq:PIII7}\deriv[2]{w}{z} = \frac1w \left(\deriv{w}{z}\right)^{2}- \frac{1}{z} \deriv{w}{z} +\frac{w^2}{z} + \frac{\b}{z} - \frac{1}{w},\end{gather}
or if $\de=0$ and $\b\ga\not=0$ set $\b=-1$ and $\ga=1$
\begin{gather}\label{eq:PIII7b}\deriv[2]{w}{z} = \frac1w \left(\deriv{w}{z}\right)^{2}- \frac{1}{z} \deriv{w}{z} +\frac{\a w^2}{z} - \frac{1}{z} +w^3;\end{gather}
\item if $\ga=\de=0$ and $\a\b\not=0$, which is known as $\mbox{P}^{(8)}_{\rm III}$, then set $\a=1$ and $\b=-1$, without loss of generality
\begin{gather*}
\deriv[2]{w}{z} = \frac1w \left(\deriv{w}{z}\right)^{2}- \frac{1}{z} \deriv{w}{z} +\frac{w^2}{z} - \frac{1}{z};\end{gather*}
\item if $\a=\ga=0$ (or equivalently $\b=\de=0$) then the equation can be solved by quadratures and has no transcendental solutions.
\end{enumerate}
In the sequel, we shall refer to equation \eqref{eq:PIII6} as \PIII\ rather than $\mbox{P}^{(6)}_{\rm III}$ since this is the generic case. Equation \eqref{eq:PIII7} is also known as the degenerate \PIII, cf.~\cite{refKV04,refKV10}. These different types of~\PIII\ were noted by Painlev\'e \cite{refPain1898}.

Similarly, \PV\ \eqref{eqPV} can be classified into three cases:
\begin{enumerate}[(i)]\itemsep=0pt
\item if $\de\not=0$, then set $\de=-\tfrac12$, without loss of generality;
\item if $\de=0$ and $\ga\not=0$, then the equation is known as degenerate \PV\ (deg-\PV)
\begin{gather}\deriv[2]{w}{z} = \left(\frac{1}{2w} + \frac{1}{w-1}\right){\left(\deriv{w}{z}\right)^{\!2}} - \frac{1}{z} \deriv{w}{z} + \frac{(w-1)^2}{z^2}\left(\a w +\frac{\b}{w}\right)+ \frac{\ga w}{z},\label{eqPV0}\end{gather}
which is equivalent to \PIII\ \eqref{eq:PIII6}, cf.~\cite[Theorem~4.2]{refFA82}, \cite[Section~34]{refGLS};
\item if $\ga=0$ and $\de=0$ then the equation can be solved by quadratures and has no transcendental solutions.
\end{enumerate}

Each of the Painlev\'e equations can be written as a Hamiltonian system
\begin{gather*}
\frac{\d q}{\d z}=\pderiv{\mathcal{H}_{\rm J}}{p},\qquad \frac{\d p}{\d z}=-\pderiv{\mathcal{H}_{\rm J}}{q},
\end{gather*}
for a suitable Hamiltonian function $\mathcal{H}_{\rm J}(q,p,z)$ \cite{refJMii,refOkamoto80a,refOkamoto80b}. The function $\sigma(z)\equiv\mathcal{H}_{\rm J}(q,p,z)$ satisfies a second-order, second-degree ordinary differential equation, known as the ``Painlev\'e $\sigma$-equation", whose solution is expressible in terms of the solution of the associated Painlev\'e equation \cite{refJMii,refOkamoto80b}. The Painlev\'e $\sigma$-equations (\sPI--\sPVI) associated with \PI--\PVI\ respectively are
\begin{gather}
\left(\deriv[2]{\sigma}{z}\right)^{\!2} + 4\left(\deriv{\sigma}{z}\right)^{\!3} +2z\deriv{\sigma}{z}-2\sigma=0,\label{eqSI}\\
\left(\deriv[2]{\sigma}{z}\right)^{\!2} + 4\left(\deriv{\sigma}{z}\right)^{\!3}
+2\deriv{\sigma}{z}\left(z\deriv{\sigma}{z}-\sigma\right)=\tfrac14{\b^2},\label{eqSII}\\
\left(z\deriv[2]{\sigma}{z}-\deriv{\sigma}{z}\right)^{\!2} + \left[4\left(\deriv{\sigma}{z}\right)^{\!2}-z^2\right]\left(z\deriv{\sigma}{z}-2\sigma\right)
+4z\vth_{\infty}\deriv{\sigma}{z}=2\vth_0z^2,\label{eqSIII} \\
\left(\deriv[2]{\sigma}{z}\right)^{\!2} - 4\left(z\deriv{\sigma}{z}-\sigma\right)^{\!2}+4\deriv{\sigma}{z}\left(\deriv{\sigma}{z}+2\vth_0\right)\left(\deriv{\sigma}{z}+2\vth_{\infty}\right)=0,\label{eqSIV} \\
\left(z\deriv[2]{\sigma}{z}\right)^{\!2} - \left[2\left(\deriv{\sigma}{z}\right)^{\!2}-z\deriv{\sigma}{z}+\sigma\right]^2+4\prod_{j=1}^4\left(\deriv{\sigma}{z}+\k_j\right)=0,\label{eqSV}\\
\deriv{\sigma}{z} \left[z(z-1)\deriv[2]{\sigma}{z}\right]^2 + \left[\deriv{\sigma}{z}\left\{2\sigma-(2z-1)\deriv{\sigma}{z}\right\}
+\k_1\k_2\k_3\k_4\right]^2 =\prod_{j=1}^4\left(\deriv{\sigma}{z}+\k_j^2\right),\label{eqSVI}
\end{gather}
where $\b$, $\vth_0$, $\vth_{\infty}$ and $\k_1,\ldots,\k_4$ are arbitrary constants.

\begin{Example}The Hamiltonian associated with \PII\ (\ref{eqPII}) is
\begin{gather}\label{sec:PT.HM.DE4}
\HII(q,p,z;\a) = \tfrac12 p^2 - \big(q^2+\tfrac12 z\big)p - (\a+\tfrac12)q
\end{gather} and so
\begin{gather}\label{sec:PT.HM.DE3}
\deriv{q}{z}=p-q^2-\tfrac12z,\qquad \deriv{p}{z}=2qp+\a+\tfrac12,
\end{gather} see \cite{refJMii,refOkamotoPIIPIV}.
Eliminating $p$ in (\ref{sec:PT.HM.DE3}) then $q$ satisfies \PII\ (\ref{eqPII}) whilst eliminating $q$ yields
\begin{gather}\label{eq:P34}
\deriv[2]{p}{z} =\frac1{2p}\left(\deriv{p}{z}\right)^{\!2} + 2p^2-zp-\frac{(\a+\tfrac12)^2}{2p},
\end{gather}
which is known as \Ptf\ since is equivalent to equation XXXIV of Chapter 14 in \cite{refInce}. Hence if~$q$ satisfies \PII\ (\ref{eqPII}) then $p=q'+q^2+\tfrac12 z$ satisfies~(\ref{eq:P34}). Conversely if $p$ satisfies (\ref{eq:P34}) then $q=\big(p'-\a-\tfrac12\big)/(2p)$ satisfies \PII\ (\ref{eqPII}). Thus there is a one-to-one correspondence between solutions of \PII\ (\ref{eqPII}) and those of \Ptf\ (\ref{eq:P34}). Further, the function $\sigma(z;\a)=\HII(q,p,z;\a)$ defined by (\ref{sec:PT.HM.DE4}), where $q$ and $p$ satisfy the system (\ref{sec:PT.HM.DE3}), then $\sigma(z;\a)$ satisfies (\ref{eqSII}). Conversely if $\sigma(z;\a)$ is a solution of (\ref{eqSII}), then
\begin{gather*}
q(z;\a)=\frac{4\sigma''(z;\a)+2\a+1}{8\sigma'(z;\a)},\qquad p(z;\a)=-2\sigma'(z;\a),
\end{gather*}
with $'\equiv {\rm d}/{\rm d}z$, are solutions of (\ref{eqPII}) and (\ref{eq:P34}), respectively \cite{refJMii,refOkamoto80a,refOkamoto80b,refOkamotoPIIPIV}.
\end{Example}

In this paper some open problems associated with the Painlev\'e equations are discussed. Specifically the following open problems are discussed.
\begin{enumerate}\itemsep=0pt
\item Develop algorithmic procedures for the Painlev\'e equivalence problem: given an equation with the Painlev\'e property, how do we know if the equation can be solved in terms of a~Painlev\'e equation (or a Painlev\'e $\sigma$-equation)?
\item Develop software for numerically studying the Painlev\'e equations which utilizes the fact that they are integrable equations solvable using isomonodromy methods.
\item Develop a notation for the Painlev\'e transcendents which takes into account the wide variety of solutions the Painlev\'e equations have.
\item Provide a complete classification and unified structure of the special properties which the Painlev\'e equations possess~-- the presently known results are rather fragmentary and non-systematic.
\end{enumerate}

\section{Painlev\'e equivalence problem}
For a \textit{linear} ordinary differential equation, if it can be solved in terms of known functions then the equation is regarded as being is solved. Symbolic software such as MAPLE can easily find the solutions of the linear ordinary differential equations, as illustrated in the following example. \begin{Example}{\rm Consider the linear equations
\begin{gather*}
\deriv[2]{v}{z}+ z^2v=0,\qquad\deriv[2]{w}{z} + \e^{2z}w=0,\end{gather*}
which respectively have the solutions
\begin{gather*}
v(z)=\sqrt{z}\left\{C_1J_{1/4}\left(\tfrac12z^2\right)+C_2J_{-1/4}\left(\tfrac12z^2\right)\right\},\qquad
w(z)=C_1J_{0}(\e^z)+C_2Y_{0}(\e^z),\end{gather*}
with $C_1$ and $C_2$ arbitrary constants, $J_{\nu}(\zeta)$ and $Y_{\nu}(\zeta)$ {Bessel functions}.}\end{Example}

It is a general property of linear ordinary differential equations that all singularities of their solutions are fixed. For example, solutions of the second-order equation
\begin{gather*}\deriv[2]{w}{z} + p(z)\deriv{w}{z} +q(z)w = 0,\end{gather*}
can only have singularities where the coefficients do, namely at the singularities of $p(z)$ and $q(z)$.

\begin{Definition}A \textit{fixed singular point} of a solution of an ordinary differential equation is a~singular point whose location does not vary with the particular solution chosen but depends only on the equation.
\end{Definition}

However it is not as simple for \textit{nonlinear} ordinary differential equations which are quite different since, in general, their solutions can have both movable and fixed singularities.

\begin{Definition} A \textit{movable singular point} of a solution of an ordinary differential equation is one whose location
depends on the constant(s) of integration. \end{Definition}

Currently there is no symbolic software available even to identify a nonlinear ordinary differential equation let alone find a solution, except for a few very simple examples. It is quite straightforward to determine whether a given (nonlinear) ordinary differential equation has the Painlev\'e property, e.g., using the Painlev\'e test \cite{refARSi,refARSii}; see also \cite{refAC,refCM08,refKC,refKJH}.

Painlev\'e, Gambier et al.\ classified all ordinary differential equations of the form \eqref{eq:gen-ode} with the Painlev\'e property, up to a M\"obius transformation \eqref{eq:mobius}. Consequently, a given equation of the form \eqref{eq:gen-ode} with the Painlev\'e property which is \textit{not} in the list of fifty equations given by Ince \cite[Chapter~14]{refInce}, how does one determine the M\"obius transformation? If the equation is autonomous, or has a symmetry, then it has a first integral and one should be able to solve it in terms of elliptic equations, linear equations or by quadratures. If the equation is non-autonomous and does not possess a symmetry then it is likely to be solvable in terms of a Painlev\'e transcendent. The question is then to which one of the Painlev\'e equations \eqref{eqPI}--\eqref{eqPVI} is the equation solvable in terms of?

We note that the solutions of some of the equations in the list given by Ince \cite[Chapter~14]{refInce} are solved in terms of Painlev\'e transcendents. For example, equation XX in the list, namely
\begin{gather*} \deriv[2]{u}{z}=\frac{1}{2u}\left(\deriv{u}{z}\right)^2 + 4u^2+zu,
\end{gather*}
is solvable in terms of \PII\ since letting $u(z)=\sqrt{w(z)}$ yields \eqref{eqPII} with $\a=0$.

\begin{Example}Consider the equation
\begin{gather}\label{Tzode}
\deriv[2]{w}{z} = \frac1w\left(\deriv{w}{z}\right)^2-\frac1z \deriv{w}{z} + w^3-1.\end{gather}
This equation can be shown to possess the Painlev\'e property, but is not in the list of fifty equations given in \cite[Chapter 14]{refInce}. Equation \eqref{Tzode} arises from the symmetry reduction
\begin{gather*} u(x,t)=\ln w(z),\qquad z=2\sqrt{xt},\end{gather*}
of the Tzitz\'{e}ica equation \cite{refTzitzeica1,refTzitzeica2,refTzitzeica3}
\begin{gather*} u_{xt}=\e^{2u}-\e^{-u},
\end{gather*}
see also \cite{refDB77,refMik79,refMik81,refZS}. Making the transformation
\begin{gather}\label{eq2}w(z)=x^{1/3}y(x),\qquad z=\tfrac32 x^{2/3},\end{gather}
in equation \eqref{Tzode} yields
\begin{gather}\label{eq3}\deriv[2]{y}{x} =\frac1y\left(\deriv{y}{x}\right)^2 -\frac1x \deriv{y}{x} + y^3-\frac1x,\end{gather}
which is the special case of $\mbox{P}^{(7)}_{\rm III}$ \eqref{eq:PIII7b} with $\a=0$. The transformation \eqref{eq2} is suggested by the asymptotic expansions of \eqref{Tzode} and \eqref{eq3}
\begin{alignat*}{3}
& w(z)\sim 1+\la z^{-1/2}\exp\big({-}\sqrt{3} z\big),\qquad && \text{as}\quad z\to\infty,& \\
& y(x)\sim x^{-1/3}+\k x^{-2/3}\exp\left(-\tfrac32\sqrt{3} x^{2/3}\right),\qquad & &\text{as}\quad x\to\infty,&
\end{alignat*}
with $\la$ and $\k$ constants. Consequently one can derive the isomonodromy problem for equation~\eqref{Tzode} from that of equation~\eqref{eq3}.
\end{Example}

\begin{Example}Consider the complex sine-Gordon equation
\begin{gather}\nabla^2\psi + \frac{(\nabla\psi)^2\overline{\psi}}{1-|\psi|^2}+\psi\left(1-|\psi|^2\right)=0,\label{eq:csg1}\end{gather}
where $\nabla\psi=(\psi_x,\psi_y)$, which is also known as the Pohlmeyer--Lund--Regge model \cite{refLund77,refLR,refPohl}. This has a separable solution in polar coordinates given by
$\psi(r,\theta)=\ph_n(r) \e^{\i n\theta}$, where $\ph_n(r)$ satisfies the second-order equation
\begin{gather}\deriv[2]{\ph_n}{r}+\frac1r \deriv{\ph_n}{r}+ \frac{\ph_n}{1-\ph_n^2}\left\{\left(\deriv{\ph_n}{r}\right)^2-\frac{n^2}{r^2}\right\}+\ph_n\big(1-\ph_n^2\big)=0,\label{eq:csg2}\end{gather}
which also arises in extended quantum systems \cite{refCFH05,refCH05,refCH08}, in relativity \cite{refGMS} and reflection coefficients for orthogonal polynomials on the unit circle \cite[equation (3.13)]{refWVAbk}.
Equation \eqref{eq:csg2} can be shown to possess the Painlev\'e property, though is not in the list of 50 equations given in \cite[Chapter 14]{refInce}.
Equation \eqref{eq:csg2} can be transformed into \PV\ \eqref{eqPV} in two different ways:
\begin{enumerate}[(i)]\itemsep=0pt
\item the transformation \begin{gather*}\ph_n(r)=\frac{1+u(z)}{1-u(z)},\qquad\text{with}\quad r=\tfrac12z,\end{gather*}
yields
\begin{gather*}\deriv[2]{u}z= \left(\frac{1}{2u} + \frac{1}{u-1}\right)\left(\deriv{u}z\right)^{2} -\frac{1}{z} \deriv{u}z+\frac{n^2(u-1)^2(u^2-1)}{8z^2u} -\frac{u(u+1)}{2(u-1)},\end{gather*}
which is \PV\ \eqref{eqPV} with $\a=\tfrac18n^2$, $\b=-\tfrac18n^2$, $\ga=0$ and $\de=-\tfrac12$;
\item
the transformation
\begin{gather*}\ds\ph_n(r)=\sqrt{\frac{v(z)}{v(z)-1}},\qquad\text{with}\quad r=\sqrt{z},\end{gather*}
yields \begin{gather*}\deriv[2]{v}{z}=\left(\frac1{2v}+\frac1{v-1}\right)\left(\deriv{v}{z}\right)^2-\frac1z \deriv{v}{z}+\frac{n^2v(v-1)^2}{2z^2}-\frac{v}{2z},\end{gather*}
which is \PV\ \eqref{eqPV} with $\a=\tfrac12n^2$, $\b=0$, $\ga=-\tfrac12$ and $\de=0$, i.e., deg-\PV\ \eqref{eqPV0}.
\end{enumerate}
It is known that deg-\PV\ \eqref{eqPV0} is equivalent to \PIII\ \eqref{eq:PIII6}, cf.~\cite[Section~34]{refGLS}. Using this it can be shown that if $w(z)$ satisfies
\begin{gather*}\deriv[2]{w}{z} = \frac{1}{w}{\left(\deriv{w}{z}\right)^{\!2}} - \frac{1}{z} \deriv{w}{z} -\frac{2n w^2}{z} + \frac{2n+2}{z} + \ga w^3 + \frac{\de}{w},\end{gather*}
which is \PIII\ \eqref{eq:PIII6} with $\a=-2n$ and $\b=2n+2$, then
\begin{gather*} \ph_n(r) = \frac{ \sqrt{-zw'(z)+zw^2(z)+(2n+1)w(z)+z}}{\sqrt{2z} w(z)},\qquad\text{with}\quad r=z,\end{gather*}
satisfies \eqref{eq:csg2}. Consequently solutions of equation \eqref{eq:csg2} can be expressed in terms of solutions of both \PIII\ \eqref{eqPIII} and \PV\ \eqref{eqPV}.
\end{Example}

The function $\ph_{n}(r)$ also satisfies the differential-difference equations
\begin{subequations}\label{eq2ab}\begin{gather}
\deriv{\ph_{n}}{r}+\frac{n}{r}\ph_{n}-\big(1-\ph_{n}^2\big)\ph_{n-1}=0,\label{eq2a}\\
\deriv{\ph_{n-1}}{r}-\frac{n-1}{r}\ph_{n-1}+\big(1-\ph_{n-1}^2\big)\ph_{n}=0.\label{eq2b}
\end{gather}\end{subequations}
Solving \eqref{eq2a} for $\ph_{n-1}(r)$ and substituting in \eqref{eq2b} yields equation \eqref{eq:csg2}, whilst eliminating the derivatives in \eqref{eq2ab}, after letting $n\to n+1$ in \eqref{eq2b}, yields the difference equation
\begin{gather}\ph_{n+1}+\ph_{n-1}={\frac{2n}{r} \frac{\ph_{n}}{1-\ph_{n}^2}},\label{csgeq3}
\end{gather}
which is known as the discrete Painlev\'e II equation \cite{refNP,refPS,refWVAbk}. If $n=1$ then equa\-tions~\eqref{eq2ab} have the solution
\begin{gather*}
\ph_0(r)=1,\qquad\ph_1(r)=\frac{C_1I_1(r)-C_2K_1(r)}{C_1I_0(r)+C_2K_0(r)},
\end{gather*} where $I_0(r)$, $K_0(r)$, $I_1(r)$ and $K_1(r)$ are the imaginary Bessel functions and $C_1$ and $C_2$ are arbitrary constants. Then one can use \eqref{csgeq3} to determine $\ph_{n}(r)$, for $n=2,3,\ldots$. Using this Barashenkov and Pelinovsky \cite{refBP98} derive explicit multi-vortex solutions for the complex sine-Gordon equation~\eqref{eq:csg1}.

The relationship between solutions of \eqref{eq:csg2} and those of \PIII\ \eqref{eq:PIII6}, is illustrated in the following theorem.

\begin{Theorem}\label{thm:4.1} If $\ph_{n}(r)$ satisfies \eqref{eq:csg2} then
$\ds w_n(r)={\ph_{n+1}(r)}/{\ph_{n}(r)}$ satisfies
\begin{gather*}\deriv[2]{w_n}{r} = \frac{1}{w_n}\left(\deriv{w_n}{r}\right)^2-\frac{1}{r}\deriv{w_n}{r} -\frac{2n}{r}w_n^2+ \frac{2n+2}{r} + w_n^3-\frac{1}{w_n},
\end{gather*}
which is {\PIII} \eqref{eqPIII} with parameters $\a=-2n$ and $\b=2n+2$.\end{Theorem}
\begin{proof} See Hisakado \cite{refHisakado} and Tracy and Widom \cite{refTW99}. \end{proof}

\begin{Example}In their study of third-order ordinary differential equations, Mu\u{g}an and Jrad \cite{refMF04} show that the equations
\begin{gather} y^{2}\deriv[3]{y}{x}=4y \deriv{y}{x}\deriv[2]{y}{x} -3 \left( \deriv{y}{x}\right)^{3}+ y^4
\deriv{y}{x} +4\k\mu x \left( y \deriv{y}{x} -{\k y^3}\right) -4\k\mu y^2+3\mu \deriv{y}{x},\label{eq:MJ267}\\
y\deriv[3]{y}{x}=2 \deriv{y}{x}\deriv[2]{y}{x} -2y^2\deriv[2]{y}{x} + y^3\deriv{y}{x}+ y^5
+\k \left(2\deriv{y}{x}+xy^3 + y^2\right),\label{eq:MJ34}\\
y\deriv[3]{y}{x}=\deriv{y}{x}\deriv[2]{y}{x} -2y^3 + \k y^2-\frac{\k^2}{12}\left(x\deriv{y}{x}-y\right),\label{eq:MJ36}
\end{gather}
where $\k$ and $\mu$ are non-zero constants, have the Painlev\'e property. In \cite{refMF04} see equation (2.67) with $k_1=3\mu$, $k_2=0$, without loss of generality, and $k_3=4\k\mu$; equation (2.106) with $k_2=\k$ and $k_3=0$, without loss of generality; and equation~(4.14) with $k_1=\k$ and $k_2=0$, without loss of generality\footnote{The sign of the last term in \eqref{eq:MJ36} has been changed as there is a sign error in \cite[equation (4.14)]{refMF04}.}. Levi, Sekera and Winternitz \cite{refLSW} show that \eqref{eq:MJ267}, \eqref{eq:MJ36} and \eqref{eq:MJ36} have no symmetries and state that these equations are ``candidates for new Painlev\'e transcendents"; see equations (3.3), (3.4) and (3.4) in \cite{refLSW}. However, we show below, equation \eqref{eq:MJ267} can be solved in terms of \PIV\ \eqref{eqPIV} and equation \eqref{eq:MJ34} and \eqref{eq:MJ36} in terms of \Ptf\ \eqref{eq:P34}.

Letting $\ds y=\frac{1}{u}\deriv{u}{x}$ in equation \eqref{eq:MJ267} gives the tri-linear equation
\begin{gather*} \left( \deriv{u}{x} \right) ^{2}\deriv[4]{u}{x} -4\deriv{u}{x} \deriv[2]{u}{x} \deriv[3]{u}{x} +3 \left( \deriv[2]{u}{x} \right) ^{3}+ \left( 4\k\mu x u\deriv{u}{x} +3\mu u^{2} \right)\deriv[2]{u}{x}\\
\qquad{} +{4\k\mu (\k +1)x}\left( \deriv{u}{x} \right) ^{3} -\mu (4\k +3)u \left( \deriv{u}{x} \right) ^{2}=0, \end{gather*}
which has the first integral, the bi-linear equation
\begin{gather*}\left( \deriv{u}{x} \right) \deriv[3]{u}{x} = \frac32 \left( \deriv[2]{u}{x} \right) ^{2} -\left(2{\k}^{2}\mu {x}^{2}+K \right)\left( \deriv{u}{x} \right)^{2}
-4 \k\mu x u\deriv{u}{x} -\frac{3}{2}\mu u^2,\end{gather*}
 with $K$ a constant of integration. Since $\ds\deriv{u}{x}=uy$, then we obtain the second-order ordinary differential equation
\begin{gather*}\deriv[2]{y}{x}=\frac{3}{2y} \left(\deriv{y}{x} \right)^{2}+\frac{1}{2}y^3 -\left({2{\k}^{2}\mu {x}^{2}}-K\right)y -4\k\mu x-{\frac {3\mu }{2y}},
\end{gather*}
which is the first integral of \eqref{eq:MJ267} and is not one of the 50 equations in Ince's list. However, making the transformation
\begin{gather*} y(x)=\frac{\mu^{1/4}}{\k^{1/2} w(z)},\qquad z=\k^{1/2}\mu^{1/4} x,\end{gather*}
yields \PIV\ \eqref{eqPIV} with parameters
\begin{gather*} \a=\frac{K}{\k\mu^{1/2}},\qquad \b=-\frac{1}{2\k^2}.\end{gather*}

Letting $\ds y=\frac{1}{u}\deriv{u}{x}$ in equation \eqref{eq:MJ34} gives the tri-linear equation
\begin{gather*} u\deriv{u}{x}\deriv[4]{u}{x}=\left[u\deriv[2]{u}{x}+\left(\deriv{u}{x}\right)^2\right]\deriv[3]{u}{x}+2\k u\deriv[2]{u}{x}+\k x\left(\deriv{u}{x}\right)^3-\k u\left(\deriv{u}{x}\right)^2,\end{gather*}
which has the first integral
\begin{gather} \deriv[3]{u}{x}+(\k x-3C_1u)\deriv{u}{x}+2\k u=0,\label{eq:P343}\end{gather}
with $C_1$ a constant of integration. Letting $u=v+\k x/C_1$ gives
\begin{gather*}\deriv[3]{v}{x}-3C_1v\deriv{v}{x}=2\k x\deriv{v}{x}+\k v.\end{gather*}
Multiplying this by $v$ and integrating gives
\begin{gather*} v\deriv[2]{v}{x}=\frac12\left(\deriv{v}{x}\right)^2+C_1 v^3+\k xv^2+C_2,\end{gather*}
with $C_2$ a constant of integration, which is equivalent to \Ptf\ \eqref{eq:P34}, by rescaling the variables if necessary. Due to the relationship between \Ptf\ \eqref{eq:P34} and \PII\ \eqref{eqPII} \cite{refFA82}, solutions of equation~\eqref{eq:MJ34} can be expressed in terms of solutions of~\PII~\eqref{eqPII}. Specifically if~$w(z)$ is a solution of~\PII~\eqref{eqPII}, then
\begin{gather*} y(x) = -\k^{1/3}\left\{2w(z) + \frac{zw(z)+\a}{w'(z)+w^2(z)}\right\},\qquad x=-z/\k^{1/3},\end{gather*}
where $'\equiv \d/\d z$ satisfies \eqref{eq:MJ34}.

Letting $\ds y=\deriv{u}{x}$ in equation \eqref{eq:MJ36} and integrating gives the third-order equation
\begin{gather*} \deriv[3]{u}{x}+(24u-\k x)\deriv{u}{x}-\frac{\k^2 x}{12}=0,\end{gather*}
where the constant of integration has been set to zero, without loss of generality.
Then making the transformation
\begin{gather*}u=-\tfrac14v-\tfrac{1}{24}\k x,\end{gather*}
yields
\begin{gather} \deriv[3]{v}{x} = (6v+2\k x)\deriv{v}{x}+\k v. \label{eq:P363}\end{gather}
Multiplying this by $v$ and integrating gives
\begin{gather*} v\deriv[2]{v}{x}=\frac12\left(\deriv{v}{x}\right)^2+2 v^3+\k xv^2+C,\end{gather*}
with $C$ a constant of integration, which is equivalent to \Ptf\ \eqref{eq:P34}, by rescaling the variables if necessary.
\end{Example}

\begin{Remark}We remark that equations \eqref{eq:P343} and \eqref{eq:P363}, after rescaling the variables, arise as a scaling reduction of the Korteweg--de Vries equation~\cite{refFA82} and as a nonclassical reduction of the Boussinesq equation~\cite{refCK}.\end{Remark}

Bureau \cite{refBureau72} (see also \cite{refBGG72,refChalk87}) has also studied the classification of second order, second degree equations
\begin{gather}
\left(\deriv[2]{w}{z}\right)^2 =F\left(w,\deriv{w}{z},z\right)+G\left(w,\deriv{w}{z},z\right)\deriv[2]{w}{z}, \label{eq:gen2ode2}
\end{gather}
where $F$ and $G$ are rational in $w$ and $\ideriv{w}{z}$ and locally analytic in $z$. Cosgrove and Scoufis~\cite{refCS} have classified all equations with the Painlev\'e property for the special case of \eqref{eq:gen2ode2} when $G\equiv0$, i.e.,
\begin{gather*}
\left(\deriv[2]{w}{z}\right)^2 =F\left(w,\deriv{w}{z},z\right), 
\end{gather*}
where $F$ is rational in $w$ and $\ideriv{w}{z}$, locally analytic in $z$ and not a perfect square. Cosgrove and Scoufis \cite{refCS} solved the equations with the Painlev\'e property in terms of the Painlev\'e transcendents, elliptic functions, and solutions of linear equations, see also \cite{refCosgrove97,refCosgrove06b,refSM97,refSM98}. Cosgrove~\cite{refCosgrove93} classified all equations that are of Painlev\'e type of the form
\begin{gather*}
\left(\deriv[2]{w}{z}\right)^m =F\left(w,\deriv{w}{z},z\right),\qquad m\geq3, 
\end{gather*}
where $F$ is rational in $w$ and $\ideriv{w}{z}$ and locally analytic in $z$ and solved the equations in terms of the first, second and fourth Painlev\'e transcendents, elliptic functions, or quadratures.

For various results on classifying classes of second-order ordinary differential equations, including Painlev\'e equations, see Babich and Bordag \cite{refBB99}, Bagderina \cite{refBag07,refBag13,refBag15b,refBag15a,refBag16}, Bagderina and Tarkhanov \cite{refBagTark}, Berth and Czichowski \cite{refBC01}, Hietarinta and Dryuma \cite{refHD}, Kamran, Lamb and Shadwick \cite{refKLS}, Kartak \cite{refKartak11,refKartak12,refKartak13,refKartak14}, Kossovskiy and Zaitsev \cite{refKZ}, Milson and Valiquette~\cite{refMV}, Valiquette~\cite{refVal} and Yumaguzhin \cite{refYum}.
Most of these studies are concerned with the invariance of second-order ordinary differential equations of the form
\begin{gather*}\deriv[2]{w}{z} = F_3(w,z)\left(\deriv{w}{z}\right)^3+F_2(w,z)\left(\deriv{w}{z}\right)^2+F_1(w,z)\deriv{w}{z}+F_0(w,z),\end{gather*}
under the point transformations of the form
\begin{gather*} w=\psi(y,x),\qquad z= \phi(y,x),\qquad \frac{\partial(\psi,\phi)}{\partial(y,x)}\equiv\pderiv{\psi}{y}\pderiv{\phi}{x}-\pderiv{\psi}{x}\pderiv{\phi}{y}\not=0.\end{gather*}

Chazy \cite{refChazy09,refChazy11}, Garnier \cite{refGarnier12} and Bureau \cite{refBureau64} have obtained partial results on
the classification of ordinary differential equations with the Painlev\'e property for third-order equations of the form
\begin{gather} \label{eq:gen3ode} \deriv[3]{w}{z}=F\left(w,\deriv{w}{z},\deriv[2]{w}{z},z\right), \end{gather}
where $F$ is rational in $w$ and its derivatives, and locally analytic in $z$. Despite the considerable length of these papers, only a very small proportion of the possible equations with the Painlev\'e property in each class were discovered. Further no new transcendents were discovered, i.e., every equation with the Painlev\'e property was shown to be solvable in terms of previously known equations, either Painlev\'e transcendents, elliptic functions or quadratures.

Most of the recent studies of Painlev\'e classification for third-order equations have concentrated on equations in the \textit{Bureau polynomial class} where the function $F$ in \eqref{eq:gen3ode} is polynomial in $w$ and its derivatives, rather than rational. Cosgrove \cite{refCosgrove00a,refCosgrove06a} classified third-order equations of this specific form with the Painlev\'e property and solved the equations in terms of the Painlev\'e transcendents, elliptic functions, solutions of linear equations or quadratures; see also \cite{refBag08,refCosgrove00b}.

\begin{Problem}Given an ordinary differential equation with the Painlev\'e property, how do we know whether it can be solved in terms of a Painlev\'e transcendent?
\end{Problem}

\section{Notation for Painlev\'e transcendents}
Uniquely amongst the functions discussed in the DLMF \cite{refDLMF}, there is no special notation for the Painlev\'e transcendents, i.e., the solutions of the Painlev\'e equations. There are several functions in the DLMF whose notation involves $P$, or a variant, e.g., $P_n^{(\a,\b)}(z)$ (Jacobi polynomials), $P_n(z)$ (Legendre polynomials), and $\wp(z)$ (Weierstrass elliptic functions). For linear equations, there are a finite number of linearly independent solutions, e.g., $\Ai(z)$ and $\Bi(z)$ for the Airy equation
\begin{gather*}\deriv[2]{w}{z}-zw=0.\end{gather*}
However, for nonlinear equations such as the Painlev\'e equations, the issue of notation is not as simple as there are numerous completely different solutions. Although second-order equations, there don't exist two ``representative solutions". What is needed is some agreed notation for the Painlev\'e transcendents. In fact, unlike the linear case when the set of all solutions is a finite dimensional vector space, the set of all solutions of a Painlev\'e equation form a transcendental structure (a foliation travelling through a fibre bundle, each fibre of which is described by an affine Dynkin diagram) without any global coordinates which could be used as natural universal markers of the solutions. Such a notation would assist in the classification of properties of Painlev\'e equations.

For example, there are several different types of solutions of \PII\ \eqref{eqPII}.
\begin{enumerate}[(i)]\itemsep=0pt
\item The general solution of \PII\ is a transcendental function for \textit{all} values of $\a$ and involves {two} arbitrary constants.
\item Suppose that $w_k(z)$ is the solution of \PII\ with $\a=0$, i.e.,
\begin{gather*} \deriv[2]{w_k}{z}=2w_k^3+zw_k,
\end{gather*}
with the asymptotic behaviour
\begin{gather} w_k(z)\sim k\Ai(z),\qquad\text{as}\quad z\to\infty,\label{P2a}\end{gather}
where $k$ is a real parameter and $\Ai(z)$ is the Airy function, which uniquely determines the solution. This family of solutions has different analytical properties on the real axis and have different asymptotic behaviours as $z\to-\infty$, depending on the parameter $k$.
\begin{itemize}\itemsep=0pt
\item If $|k|<1$, then $w_k(z)$ is the \textit{Ablowitz--Segur solution} \cite{refAS77,refSA}, which is pole-free on the real axis
and as $z\to-\infty$ has oscillatory behaviour with algebraic decay given by
\begin{gather} w_k(z)= d|z|^{-1/4}\sin\left(\tfrac23|z|^{3/2} - \tfrac34 d^2\ln|z| -\theta_0\right)+o(|z|^{-1/4}),
\label{P2b}\end{gather}
where the connection formulae $d^2(k)$ and $\theta_0(k)$, which relate the asymptotic behaviours \eqref{P2a} and \eqref{P2b} as $z\to\pm\infty$, are
\begin{gather*}d^2(k) = - \pi^{-1}\ln\big(1-k^2\big),\\
\theta_0(k) = \tfrac32d^2 \ln 2 + \arg\left\{\Gamma\left(1-\tfrac12\i d^2\right)\right\} + \tfrac14\pi[1-2\sgn(k)],
\end{gather*}
see \cite{refBCLM,refCMcL,refDZ95}.
\item If $k=\pm1$ then $w_k(z)$ is the \textit{Hastings--McLeod solution} \cite{refHMcL} which is monotonic, pole-free on the real axis and has algebraic growth as $z\to-\infty$ given by
\begin{gather} w_{\pm1}(z)= \pm\left( \tfrac12 |z|\right)^{1/2}+o\big(|z|^{1/2}\big). \label{P2c}\end{gather}
\item If $|k|>1$ then $w_k(z)$ is a singular solution which has infinitely many poles on the negative real axis~-- see the numerical plot by Fornberg and Weideman \cite[Fig.~12]{refFW14}~-- and has singular oscillatory behaviour as $z\to-\infty$ given by
\begin{gather} w_k(z)= \frac{\sqrt{|z|}}{\sin\big\{\tfrac23|z|^{3/2}+\beta\ln\big(8|z|^{3/2}\big)+\phi\big\}+\O\big(|z|^{-3/2}\big)}+\O\big(|z|^{-1}\big),\label{P2d}\end{gather}
where $z$ bounded away from the singularities appearing in the denominator and the connection formulae
$\beta(k)$ and $\phi(k)$ are
\begin{gather*} \b(k)=\tfrac{1}{2}\pi^{-1}\ln\big(k^2-1\big),\qquad \phi(k) =-\arg\big[\Gamma\big(\tfrac12\i\b\big)\big]+\tfrac12\pi[\sgn(k)-1],\end{gather*}
see \cite{refBothIts,refKap}.
\item Bothner \cite{refBoth17} discusses the transition from the Ablowitz--Segur solution \eqref{P2b} and the singular solution \eqref{P2d} to the Hastings--McLeod solution \eqref{P2c} as $z\to-\infty$ and $|k|\to1$. The transition asymptotics are expressed in terms of the Jacobi elliptic functions.
\item The case when $k=\i\k$, with $\k\in\R$, in the boundary condition \eqref{P2a}, known as the \textit{pure imaginary Ablowitz--Segur solution}, is discussed by Its and Kapaev \cite{refIK03}.
\item Bogatskiy, Claeys and Its \cite{refBCI} extended these results to discuss \textit{complex Ablowitz--Segur solutions} in the case when $k\in\C$.
\end{itemize}
\item For \PII\ with $\a\not=0$, there are analogs of the Ablowitz--Segur and Hastings--McLeod solutions, known as the \textit{quasi-Ablowitz--Segur solution} and the \textit{quasi-Hastings--McLeod solution} \cite{refCKV,refDH17,refDH18}; see also \cite{refCDKSV,refFW14,refForrW15,refTroy}.
There is an extensive literature regarding the asymptotics for \PII\ \eqref{eqPII} when $\a=0$. There are significantly fewer asymptotic results in the case when $\a\not=0$. Further, whilst the Ablowitz--Segur and Hastings--McLeod solutions have exponential decay as $z\to\infty$ given by \eqref{P2a}, when $\a\not=0$ then the solutions only have algebraic decay given by
\begin{gather*} w(z;\a) \sim -\a/z,\qquad\text{as}\quad z\to\infty.\end{gather*}
For the quasi-Ablowitz--Segur solution when $\alpha\in\big({-}\tfrac12,\tfrac12\big)$, there exists a one-parameter family of real solutions~$w(z)$ for $k\in(-\cos(\pi\alpha),\cos(\pi\alpha))$ with the following properties:
\begin{gather*}
w(z)=B(z;\alpha)+k\Ai (z)\big[1 +\mathcal{O}\big(z^{-3/4}\big)\big], \qquad\text{as}\quad z\to\infty,
\end{gather*}
and
\begin{gather*}
w(z)=d|z|^{-1/4} \cos\left(\tfrac23|z|^{3/2}-\tfrac34d^2\ln|z|+\phi\right)+\mathcal{O}\big(|z|^{-1}\big), \qquad\text{as}\quad z\to-\infty,
\end{gather*}
where $\Ai (z)$ is the Airy function and $B(z;\alpha)$ is given by
\begin{gather*}
B(z;\alpha)\sim\frac{\alpha}{z}\sum_{n=0}^\infty \frac{a_n}{z^{3n}},
\end{gather*}
with coefficients $a_n$ which are uniquely determined by the recurrence relation
\begin{gather*}
a_{n+1} = (3n +1)(3n + 2)a_n-2\alpha^2\sum_{j,k,\ell=0}^n a_ja_ka_{\ell},\qquad a_0=1.
\end{gather*}
The connection formulas are given by
\begin{gather*}
d(k) =\pi^{-1/2}\sqrt{-\ln\big(\cos^2(\pi\alpha)-k^2\big)},\\
\phi(k)=-\tfrac32d^2\ln 2 + \arg\Gamma\big(\tfrac12 \i d^2\big)-\tfrac14\pi-\arg(-\sin(\pi\alpha)-k\i),
\end{gather*}
see Dai and Hu \cite{refDH17,refDH18}. For the quasi-Hastings--McLeod solution, Claeys, Kuijlaars and Vanlessen \cite{refCKV} show that there exists a unique solution which is pole-free on the real axis with the asymptotic behaviours
\begin{alignat*}{3}
& w(z) \sim -\alpha/z,\qquad && z\to+\infty,& \\
& w(z) \sim \sqrt{\tfrac12|z|},\qquad && z\to-\infty;&
\end{alignat*}
see also \cite{refDH17,refDH18}.

\item Special function solutions of \PII\ arise if and only if $\a=n+\tfrac12$, with $n\in\mathbb{Z}$, which involve {one} arbitrary constant \cite{refGambier10}. These are expressed in terms of the $n\times n$ Wronskian determinant
\begin{gather*} \tau_n(z;\vth)=\det\left[ \deriv[j+k]{}{z}\ph(z;\vth)\right]_{j,k=0}^{n-1},\qquad n\geq 1,\end{gather*}
where
\begin{gather*} \ph(z;\vth)=\cos(\vth)\Ai(\zeta)+\sin(\vth)\Bi(\zeta),\qquad \zeta=-2^{-1/3}z,\end{gather*}
with $\Ai(\zeta)$ and $\Bi(\zeta)$ the Airy functions and $\vth$ an arbitrary constant; see also the recent studies \cite{refPAC16,refDeano18}.
\item Rational solutions of \PII\ exist if and only if $\a=n$, with $n\in\mathbb{Z}$, which involve {no} arbitrary constants \cite{refVor,refYab}. These solutions are expressed in terms of polynomials $Q_n(z)$ of degree $\tfrac12n(n+1)$, now known as the \textit{Yablonskii--Vorob'ev polynomials}, which are defined through the recurrence relation (a second-order, bilinear differential-difference equation)
\begin{gather*}
Q_{n+1}{Q_{n-1}}= zQ_{n}^2 -4\left[
Q_{n}\deriv[2]{Q_{n}}{z}-\left(\deriv{Q_{n}}{z}\right)^2\right],\end{gather*}
with $Q_{0}(z)=1$ and $Q_{1}(z)=z$.
Clarkson and Mansfield \cite{refCM03} investigated the locations of the roots of the Yablonskii--Vorob'ev polynomials in the complex plane and showed that these roots have a very regular, approximately triangular structure; the term ``approximate" is used since the patterns are not exact triangles as the roots lie on arcs rather than straight lines. Bertola and Bothner \cite{refBB15} and Buckingham and Miller \cite{refBM14,refBM15} have studied the Yablonskii--Vorob'ev polynomials $Q_n(z)$ in the limit as $n\to\infty$ and shown that the roots lie in a ``triangular region'' with elliptic sides which meet with interior angle $\tfrac25\pi$.
\item There exist {tronqu\'{e}e} and {tri-tronqu\'{e}e} solutions of \PII, which are pole-free in sectors of the complex plane \cite{refBou13,refBou14}; see also \cite{refBertola,refHXZ,refJK,refJM,refMiller,refNovok12}.
\end{enumerate}
\begin{Problem}
Develop a notation for the Painlev\'e transcendents which takes into account the wide variety of solutions the Painlev\'e equations have.
\end{Problem}
\section{Numerical solution of Painlev\'e equations}
Numerical analysis of the Painlev\'e equations presents novel challenges: in particular, in contrast to the classical special functions, where the linearity of the equations greatly simplifies the situation, each problem for the nonlinear Painlev\'e equations arises essentially anew. Ideally what is needed is reliable, easy to use software to compute numerically the solutions of the Painlev\'e equations. On the other hand, Painlev\'e transcendents, being solutions of \textit{integrable} nonlinear equations, have much global information available about them. The software should be in the form of a living document where new numerical problems can be addressed by a pool of experts as they arise, as well as providing access to existing software. At the technical level, how does one combine asymptotic information about the solutions obtained from the Riemann--Hilbert problem, together with efficient numerical codes in order to compute the solution $w(z)$ at finite values of $z$? A comprehensive analysis presents many challenges, conceptual, philosophical and technical.

Deift \cite{refDeift08} wrote:
\begin{quote}{\it Writing useful numerical software for such nonlinear equations $[$i.e., the Painlev\'e equations$]$ presents many challenges, conceptual, philosophical and technical. Without the help of linearity, it is not at all clear how to select a broad enough class of ``representative problems".
}\end{quote}

Numerical simulations of the Painlev\'e equations given in \cite{refPAC05review,refPAC16} were obtained using MAPLE using the \texttt{DEplot} command with option \texttt{method=dverk78}, which finds a numerical solution using a seventh-eighth order continuous Runge--Kutta method. This is relatively simple to use, gives plots of solutions quickly with accuracy better than the human eye can detect, and generally works fine for initial value problems.

Some recent numerical computations of Painlev\'e equations include: a pole field solver using Pad\'e approximations \cite{refFFW1,refFFW2,refFW11,refFW14,refFW15,refRF13,refRF14}; numerical Riemann--Hilbert problems \cite{refOlver11,refOlver12,refOT14a,refOT14b,refTO,refWB}; Fredholm determinants \cite{refBorn1,refBorn2}; Pad\'e approximations \cite{refNovok09,refNovok14,refYamada}; pole elimination \cite{refAY12a,refAY12b,refAY12c,refAY13c,refAY13a,refAY13b,refAY15};
a multidomain spectral method \cite{refKS}.

\begin{Problem}\quad
\begin{itemize} \itemsep=0pt
\item The Runge--Kutta method, and variants, are standard ODE solvers. Can we do better for integrable equations such as the Painlev\'e equations?
\item Painlev\'e equations are ``integrable" and solvable by the isomonodromy method through an associated Riemann--Hilbert problem. How can we use this in the development of software for studying the Painlev\'e equations numerically?
\item It is well known that there are {discrete Painlev\'e equations}, which are integrable discrete equations that tend to the associated Painlev\'e equations in an appropriate continuum limit. Should we use a ``integrable discretization" of the Painlev\'e equations?
\end{itemize}
\end{Problem}

\section{Classification of properties of Painlev\'e equations}
The Painlev\'e equations are known to have a cornucopia of properties such as: a Hamiltonian representation; exact solutions (rational solutions, algebraic solutions, solutions in terms of classical special solutions);
B\"acklund transformations (which relate two solutions of a Painlev\'e equation); associated isomonodromy problems (which are Lax pairs that express the Painlev\'e equation as the compatibility of two linear systems);
and asymptotic approximations in the complex plane, with associated connection formulae relating the asymptotics. For details see \cite{refPAC05review,refConte,refFIKN,refGLS,refIKSY,refKNY,refNoumi} and the references therein.
\begin{Problem}
A complete classification and a unifying structure for these properties is needed as the presently known results are rather fragmentary and non-systematic.
\end{Problem}

\subsection*{Acknowledgments}
I would like to thank Mark Ablowitz, Andrew Bassom, Chris Cosgrove, Alfredo Dea\~{n}o, Percy Deift, Marco Fasondini, Bengt Fornberg, David G\'omez-Ullate, Rod Halburd, Andrew Hone, Alexander Its, Kerstin Jordaan, Nalini Joshi, Erik Koelink, Martin Kruskal, Ana Loureiro, Elizabeth Mansfield, Marta Mazzocco, Bryce McLeod, Peter Miller, Walter Van Assche, and Andr\'e Weideman for their helpful comments and illuminating discussions.


\pdfbookmark[1]{References}{ref}
\LastPageEnding

\end{document}